\documentclass[leqno,a4paper]{article}

\usepackage[utf8]{inputenc}
\usepackage{amsmath}
\usepackage{amsthm}
\usepackage{amsfonts}
\usepackage{amssymb}
\usepackage{mathrsfs}
\usepackage{bm}
\usepackage[dvips]{graphicx}
\usepackage{psfrag}
\usepackage{subcaption}
\usepackage{color,colordvi}
\usepackage{rotating}
\usepackage{xspace}
\usepackage{euscript}

\usepackage{amscd}
\usepackage{tabularx}
\usepackage{enumerate}

\newcommand{\C}{\mathbb{C}}

\newcommand{\R}{\mathbb{R}}
\newcommand{\N}{\mathbb{N}}

\newcommand{\modulo}[1]{\ensuremath{\left| \, #1 \, \right|}} 

\newcommand{\He}{\ensuremath{\mathrm{H}}} 

\newcommand{\ee}{\mathrm{e}}

\newcommand{\pardual}[2]{\langle #1 \,,\,  #2 \rangle}

\newcommand{\parentesis}[1]{\left( #1 \right)}

\newcommand{\conjunto}[1]{ \left\{ #1 \right\} }
\newcommand{\tq}{\ensuremath{\, : \,}} 

\newcommand{\Sobo}[2]{\ensuremath{\mathrm{H}^{#1}(#2)}} 
\newcommand{\HU}[1]{\ensuremath{\mathrm{H}^{1}_{0}(#1)}} 
\newcommand{\HZ}[2]{\ensuremath{\mathrm{H}^{#1}_{0}(#2)}} 
\newcommand{\Dlinha}{\ensuremath{{\mathcal{D}}^{\prime}}} 
\newcommand{\Fteste}[1][\R^{n}]{\ensuremath{\mathcal{C}^\infty_c(#1)}} 
\newcommand{\Le}[1]{\ensuremath{\mathrm{L}^{#1}}} 
\newcommand{\lle}[1]{\ensuremath{\ell^{#1}}} 
\newcommand{\Cont}[1]{\ensuremath{\mathcal{C}^{#1}}} 
\newcommand{\Schwartz}{\ensuremath{\mathscr{S}}}

\newcommand{\doispontos}{ \,:\, } 

\newcommand{\Fourier}{\ensuremath{\mathscr{F}}}

\newcommand{\diff}{  \ensuremath{\mathrm{d}}   }
\newcommand{\GSol}{  \ensuremath{\mathrm{G}}   }
\newcommand{\sinc}{  \ensuremath{\mathrm{sinc}}   }

\newtheorem{theorem}{Theorem}[section]
\newtheorem{lemma}[theorem]{Lemma}
\newtheorem{corollary}[theorem]{Corollary}

\theoremstyle{definition}
\newtheorem{definition}[theorem]{Definition}

\theoremstyle{remark}
\newtheorem{remark}[theorem]{Remark}

\providecommand{\keywords}[1]
{
	\small	
	\textbf{\textit{Keywords---}} #1
}

\numberwithin{equation}{section}

\title{Uniqueness result for almost periodic distributions depending on time and space and an application to the unique continuation for the wave equation}
\author{Alexandre Kawano\\
University of S\~ao Paulo, S\~ao Paulo, Brazil\\
akawano@usp.br}
\date{ }

\begin{document}
\maketitle

\begin{abstract}
	Let $\Omega\subset \R^N$, $N\geq 1$, be an open bounded and connected set with continuous piecewise $\Cont{\infty}$ boundary. 	Here we deal with almost periodic distributions of the form $u(t,x)=\sum_{n=0}^{+\infty} c_n S_n(x) \ee^{i \lambda_n t}$ where $(c_n)_{n\in\N}\subset \C$ belong to the space of slowing growing sequences $s^\prime$, and $(\lambda_n^2)_{n\in\N}\subset \R$ and $(S_n)_{n\in\N}\subset \HZ{1}{\Omega}$ are respectively the eigenvalues and eigenvectors of the Laplacian. Given $\omega\subset\Omega$, we prove that there exists $T_{max}(\Omega,\omega)>0$ depending only on $\Omega$ and $\omega$ such that  if $T>T_{max}(\Omega,\omega)$ and  $u|_{\omega\times ]-T,T[}=0$,  then $u\equiv 0$. Using this result we prove a unique continuation property for the wave equation.
\end{abstract}


\section{Introduction}
\label{Intro}

In \cite{Kawano2011}, almost periodic distributions of the type
\begin{equation}\label{eq:standardForm}
w(t)=\sum_{n=0}^{+\infty} c_n \ee^{i \mu_n t},
\end{equation}
where considered, for $(c_n)_{n\in\N}\subset s^\prime$, the space of slowly growing sequences in $\C$, meaning  there is a $q\in \N$ such that $(n^{-q}\,c_n)_{n\in\N}\in \lle{1}$, and $(\mu_n)_{n\in\N} \subset \R$ such that  there are $n_0\in\N$, $C> 0$ and $\alpha> 0$ satisfying $n\geq n_0\Rightarrow \modulo{\mu_n}\geq C n^\alpha$.

The Theory of Almost Periodic Functions was initiated by  Harald Bohr by the publication of three articles in the volumes 45, 46 and 47 of Acta Mathematica in the years 1925 and 1926. Later H. Bohr exposed his theory in a book \cite{AlmostPeriodicBohr}, and a modern view can be found in  \cite{AlmostPeriodicWaves}. H. Bohr showed that the set of all almost periodic functions is equal to the closure in $\Cont{0}$ of the set of functions
\begin{displaymath}
A=\conjunto{\R_+\ni t\mapsto\sum_{n=1}^N c_n \ee^{i \mu_n t} \tq N\in \N, c_n\in \C, \mu_n\in \R}.
\end{displaymath}

There are other definitions of Almost Periodic Functions and respective spaces, as can be seen for instance in \cite{AlmostPeriodicWaves}.  A very ``small" space is $\mathrm{AP}_1$, whose elements are of the form \eqref{eq:standardForm} with $(c_n)_{n\in\N}\in \lle{1}$. 
 
 Here we consider that the coefficients $c_n$, $n\in\N$, that appear in  \eqref{eq:standardForm} depend on another variable $x$, as detailed below.

Let $\Omega\subset \R^N$, $N\geq 1$, a bounded open connected set, with continuous and piecewise $\Cont{\infty}$ boundary. Consider the eigenvalue problem for $S\in \HZ{1}{\Omega}$, given by
\begin{equation}\label{eq:EigenProblem}
\triangle S=-\lambda^2 S, \quad \text{in $\Omega$}.\\
\end{equation} 

It is well known that this problem renders a sequence of eigenvalues $(\lambda_n^2)_{n\in\N}$ (for $N\geq 2$, with repeated values according to their geometric multiplicity) with corresponding sequence of eigenvectors $(S_n)_{n\in\N}$. We may suppose that they are normalized so that $\pardual{S_m}{S_n}_{\Le{2}(\Omega)}=1$ if $m=n$ and  $\pardual{S_m}{S_n}_{\Le{2}(\Omega)}=0$ if $m\neq n$. Note that $\pardual{S_m}{S_m}_{\HZ{1}{\Omega}}=\lambda_m^2$. Moreover, by the Weyl's law \cite{Grebenkov2013} \cite{CourantHilbert}, 
\begin{equation}
\lambda^2_{n} \propto \frac{4 \pi^{2}}{\left(\omega_{N} \tilde{\mu}_{N}(\Omega)\right)^{2 / d}} n^{2 / N},
\end{equation}
as $n\rightarrow +\infty$, where $\tilde{\mu}_{N}$ is the Lebesgue measure of $\Omega$ and $\omega_{N}=\frac{\pi^{N / 2}}{\Gamma(N / 2+1)}$.

Consider now the series 
\begin{equation}\label{eq:mainSeries}
u(t,x)=\sum_{n=0}^{+\infty} a_n S_n(x) \ee^{i \lambda_n t},
\end{equation}
with $(a_n)_{n\in\N}\subset s^\prime$, so that there is a $q\in \N$ such that $(n^{-q}\,a_n)_{n\in\N}\in \lle{1}$; $(\lambda_n)_{n\in\N}$ and $(S_n)_{n\in\N}$ are related to the eigenproblem \eqref{eq:EigenProblem}. To state it clearly, here we take $\lambda_n=\sqrt{\lambda_n^2}>0$. As it is going to be clear in Corollary \ref{teo:CorolarioTodosLambdas}, it is immaterial if we take the positive or negative real square root of $\lambda_n^2$. 

Since for fixed $\Omega$, there is a $K_{\Omega}>0$ such that $|S_n|> K_{\Omega} \lambda^2_n$ \cite{Grebenkov2013}, for any $x\in\Omega$, $u(\cdot,x)$ is a distribution lying in $\Schwartz^\prime$, that is, in the Schwartz space of tempered distributions, interpreted in the following way
\begin{equation}\label{eq:DefU}
(u(\cdot,x),\phi)=\sum_{n\in\N}\frac{a_n (-1)^p S_n(x)}{(i\,\lambda_n)^p}\int_\R \phi^{(p)}(\xi)\, \ee^{i \lambda_n \xi} \, \diff \xi, \quad \forall \phi\in \Schwartz,
\end{equation}
where $p\in\N$ is any number such that $p \geq 2+q N$. On the other hand, for each $t\in \R$, $u$ is a distribution in $\Dlinha(\Omega)$, in the following sense
\begin{displaymath}
\pardual{u(t,\cdot)}{\varphi}=\sum_{n=0}^{+\infty} a_n \pardual{S_n}{\varphi} \ee^{i \lambda_n^p t}, \, \forall \varphi \in \Fteste[\Omega],
\end{displaymath}
whose absolute convergence is guaranteed because $(\pardual{S_n}{\varphi})_{n\in \N}$ decreases faster than any polynomial and $(a_n)_{n\in\N}\subset s^\prime$.

Conditions for the unique and non-unique determination of the series \eqref{eq:standardForm} were analyzed in \cite{Kawano2011}. Now, we explore conditions where neither the information regarding the $t$-variable gathered in the set $\{ (u(x,\cdot),\varphi) : \varphi\in \Fteste[{]-T_0,T_0[}]\}$, nor in the set $\{ (u(\cdot, t),\varphi) : \varphi\in \Fteste[\Omega]\}$  regarding the $x$ variable is individually sufficient for the unique determination of the coefficients in \eqref{eq:mainSeries}, but suitable combinations of both information are.

The main result, with is obtained by the use of spherical means \cite{JohnPlaneWaves}  \cite{Kawano2013}, is stated and proved in section \ref{sec:Uniqueness}. As an application of the the results obtained here concerning the unique determination of  \eqref{eq:mainSeries}, we prove a unique continuation property for the wave equation in section \ref{sec:ApplicationWave}.

\section{The main result}\label{sec:Uniqueness}

We briefly give three simple definitions before stating our main result, which is Theorem \ref{theo:Main}.
 
\begin{definition}
	Given two points $P_1, P_2 \in \Omega\subset \R^N$, $N\in \N$, $\Omega$ bounded open and connected, using the euclidean metric, let $gd(P_1,P_2)$ be the geodesic distance between them considering only paths contained in $\Omega$, as illustrated in figure \ref{fig:Geodesic}. 
\end{definition}

\begin{figure}
	\centering
	\psfrag{1}{$P_1$}
	\psfrag{2}{$P_2$}
	\includegraphics[width=0.25\linewidth]{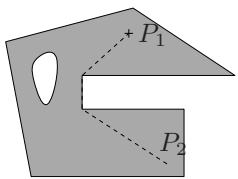}
	\caption{Geodesic distance}
	\label{fig:Geodesic}
\end{figure}

\begin{definition}
	Given a point $P \in \Omega$ and a non empty set $\Omega_0\subset \Omega$, the geodesic distance $gd(P,\Omega_0)$ is defined as
	\begin{displaymath}
	gd(P,\Omega_0)=\inf\{gd(P,Q) : Q\in \Omega_0 \}.
	\end{displaymath}
\end{definition}

\begin{definition}
	Given a non empty set $\Omega_0\subset \Omega$, $T_{max}(\Omega, \Omega_0)$ is defined as
	\begin{displaymath}
	T_{max}(\Omega,\Omega_0)=\sup\{gd(P,\Omega_0) : P\in \Omega\}.
	\end{displaymath}
\end{definition}

The result we seek to prove is the following theorem.
\begin{theorem}\label{theo:Main} 
	Let $\Omega\subset \R^N$, $N\geq 1$, be an open bounded and connected set with a continuous and piecewise $\Cont{\infty}$ boundary. Given any open set $\omega\subset\Omega$, if $T>T_{max}(\Omega,\omega)$ and  $u|_{\omega\times ]-T,T[}=0$, for $u$ given by \eqref{eq:mainSeries}, then $u\equiv 0$.
\end{theorem}

For the proof we need several ingredients. 


It is a well know fact that for any function $h:\R^N\rightarrow \R$, $h\in \Cont{1}$, it holds that
\begin{displaymath}
\frac{\partial}{\partial
	\dot{x}_{i}}\int_{B_{R}(0)}h(\dot{x}+x)\,\diff x=
\int_{B_{R}(0)}\frac{\partial}{\partial
	x_{i}}h(\dot{x}+x)\,\diff x=
\int_{\partial B_{R}(0)}h(\dot{x}+x) \nu_{i}\,\diff \sigma(x),
\end{displaymath}
where $\nu_{i}$ is the $i-th$ component of the normal to $\partial B_{R}(0)$ pointing outside of the ball. It is just Green's Theorem. By using the Stone--Weierstrass approximation Theorem, and uniform convergence in integrals, we can readily see that the above equality is true even for $h\in \Cont{0}$ when we consider only the two extremes. For $N=1$, the ball $B_R(0)$ is just a line segment centered at $x=0$, and the last integral is interpreted as usual: $\int_{\partial B_{R}(0)}h(\dot{x}+x) \nu_{i}\,\diff S_x = h(\dot{x}+R)-h(\dot{x}-R)$.

For technical reasons we  deal  the case $N=1$ separately.

\begin{lemma} \label{teo:AchadoZalcmanCasoN1}
	Let  $D$ an open interval contained in $\R$ and  $h\in \Cont{}(D)$. 
	Suppose that for all $x$ in a a non empty neighborhood $V_{x_0}\subset D$ of $x_0\in D$, and  $\forall \xi >  0$ such that $]x-\xi,x+\xi[ \subset D$, it  is true that
	\begin{equation}\label{eq:MeanN1}
	h(x-\xi)+h(x+\xi )=0.
	\end{equation}
	Then  for each $R>0$ such that $[x_0-R,x_0+R]\subset D$, 
	$h|_{[x_0-R,x_0+R]}=0$.
\end{lemma}
\begin{proof}
	The proof is by contradiction. Take $x\in V_{x_0}$ and $\xi>0$ such that $]x-\xi,x+\xi[ \subset D$. Suppose that $h(x)\neq 0$. 
	
	By hypothesis, $h(x-\xi)=-h(x+\xi )$. Now, also by hypothesis, $h(x)=h(x+\xi/2-\xi/2)=-h(x+\xi/2+\xi/2)=-h(x+\xi )$.
	
	Therefore, $h(x)=h(x-\xi)$. By symmetry, we also get $h(x)=h(x+\xi)$. But, then $h(x-\xi)+h(x+\xi)=2 h(x)\neq 0$, which contradicts \eqref{eq:MeanN1}.
\end{proof}

Now we deal with the case $N\geq 2$. The next lemma is from  Zalcman \cite{ZalcmanAnalycity}.

\begin{lemma}\label{teo:GreenModif}
	Let  $D\subset \R^N$, $N\geq 2$, be a non empty open connected set.
	Let $h:D \rightarrow \R$ a continuous function in $D$.
	
	Let $\dot{x}=\parentesis{\dot{x}_1, \dot{x}_2, \ldots , \dot{x}_N}\in D$,
	and take $R >  0$ such that $\modulo{x-\dot{x}}\leq R$ implies $x\in D$.
	
	Then 
	\begin{eqnarray*}
		\int_{\partial B_{1}(0)} h(\dot{x}&+&\xi
		R)(\dot{x}_{i}+\xi_{i}
		R)\,\diff \sigma(\xi)=
		\dot{x}_{i}\int_{\partial B_{1}(0)}
		h(\dot{x}+\xi R)\,\diff \sigma(\xi) +\\
		&&\frac{1}{R^{N-2}}\frac{\partial}{\partial
			\dot{x}_{i}}\int_{0}^R \int_{\partial B_{1}(0)}
		h(\dot{x}+\xi r)\,\diff \sigma(\xi)\,r^{N-1}\diff r,  \,	\forall i=1,\ldots, n.
	\end{eqnarray*}
\end{lemma}
 
\begin{proof}
	By Green's Theorem,
	\begin{displaymath}
	\frac{\partial}{\partial
		\dot{x}_{i}}\int_{B_{R}(0)}h(\dot{x}+x)\,\diff x=
	\int_{\partial B_{R}(0)}h(\dot{x}+x) \nu_{i}\,\diff \sigma,
	\end{displaymath}
	where $\nu_{i}$  is the $i-th$ component of the normal to $\partial B_{R}(0)$ pointing outside of the ball $\partial B_{R}(0)$.
	
	Writing this equation in polar coordinates and dividing the result by 
	$R^{n-2}$, we obtain 
	\begin{displaymath}
	\frac{1}{R^{N-2}} \frac{\partial}{\partial
		\dot{x}_{i}}\int_{0}^R \int_{\partial B_{1}(0)}
	h(\dot{x}+\xi r)\,\diff \sigma(\xi)\,r^{N-1}\diff r=
	\int_{\partial B_{1}(0)}h(\dot{x}+\xi R) R \xi_{i}\,\diff
	\sigma(\xi).
	\end{displaymath}
	
	Adding to both sides $\dot{x}_{i}\int_{\partial
		B_{1}(0)}h(\dot{x}+\xi R)\,\diff \sigma(\xi)$, we obtain the desired result.
\end{proof}

\begin{lemma} \label{teo:AchadoZalcman}
	Let  $D$ and $h$ as in the previous lemma. 
	Suppose that for all $x$ in a a non empty neighborhood $V_{x_0}\subset D$ of $x_0\in D$ and  $\forall R_x >  0$ such that $x+\xi R_x\in D$, 
	$\forall \xi\in \partial B_{1}(0)$ is true that
	\begin{displaymath}
		\int_{\partial B_{1}(0)}h(x+\xi R_x)\,\diff \sigma(\xi)=0.
	\end{displaymath}
Then  for each $R>0$ such that $B_R(x_0)\subset D$, 
	$h|_{B_R(x_0)}=0$.
\end{lemma}
\begin{proof}
	Lemma 	\ref{teo:GreenModif} implies that 
	$g_i\doispontos D \rightarrow \R$ defined by 
	$g_i(x)=h(x)x_{i}$, $i=1,\ldots, n$,  
	satisfies 
	\begin{displaymath}
		\int_{\partial B_{1}(0)}g_i(x+\xi R)\,\diff \sigma(\xi)=0, \,
		\forall x \in V_{x_0},\, \forall R  >  0, \,
		B_R(x)\subset D.
	\end{displaymath}
	
	Consequently, by induction,
	\begin{displaymath}
		\int_{\partial B_{1}(0)}h(x+\xi R)\,\mathbb{P}(x+\xi
		R)\,\diff \sigma(\xi)=0, \, 
		\forall x \in V_{x_0},\, \forall R  >  0, \,
		B_R(x)\subset D,
	\end{displaymath}
	for any polynomial function $\mathbb{P}$.
	
	By the Stone-Weierstrass Theorem, we can approximate uniformly $h(\cdot)$ over the compact $\partial B_R(x)$ by a sequence of polynomials in 
	$\partial B_{R}(x)$. Therefore,  
	\begin{displaymath}
		\int_{\partial B_{1}(0)}h^2(x+\xi R)\,\diff \sigma(\xi)=0.
	\end{displaymath}
	
	From this, we conclude that $h=0$ on any ball centered at $x\in V_{x_0}$ whose radius $R$ is such that $B_{R}(x)\subset D$.  This concludes the proof.
\end{proof}

\begin{lemma}\label{teo:SphericalMeanEigenvector}
	The spherical mean of $S_n$ around the sphere $\partial B_r(x)\subset \Omega$ centered at $x$,  whose radius $r$ is such that $B_{r}(x)\subset \Omega$, 
	\begin{displaymath}
\Phi(x, r) \doteq \frac{1}{N \alpha(N) r^{N-1}} \int_{\partial B_r(x)} S_n(y) \, \diff \sigma(y)
	\end{displaymath}
	is given by the solution of the problem
	\begin{equation}\label{eq:EqDiffPhiEnunciado}
	\begin{cases}
	r \frac{\partial^2\Phi}{\partial r^2}(x,r) + (N-1) \frac{\partial\Phi}{\partial r}(x,r)+ \lambda_n^{2} r \Phi(x,r)=0 & \\
	\Phi(x,0)= S_n(x), &\\
	\frac{\partial\Phi}{\partial r}(x,0)=0. & 
	\end{cases}
	\end{equation}
in the region $\{r\in\R : B_{|r|}(x)\subset \Omega\}$.
\end{lemma}
\begin{proof}
	Let us take the spherical mean of $S_n$ over the sphere $\partial B_r(x)\subset \Omega$ centered at $x$.
\begin{equation}\label{eq:SphericalMean}
\begin{split}
\Phi(x,r)&= \frac{1}{N \alpha(N) r^{N-1}} \int_{\partial B_r(x)} S_n(y) \, \diff \sigma(y) \\
&= \frac{1}{N \alpha(N)} \int_{\partial B_1(0)} S_n(x+r z) \, \diff \sigma(z).
\end{split}
\end{equation}
Note that this formula is valid for all $r\in\R$ such that $B_{|r|}(x)\subset \Omega$, and for fixed $x$, it is an even function of $r$.

Taking its first derivative with respect to $r$, we obtain
\begin{displaymath}
\frac{\partial\Phi}{\partial r}(x,r)= \int_{\partial B_1(0)} D S_n(x+r z) \cdot z \frac{\diff \sigma(z)}{N \alpha(N)}= \frac{1}{N r^{N-1} \alpha(N)} \int_{B_r(x)} \triangle S_n(y)\, \diff y.
\end{displaymath}

But $\triangle S_n=-\lambda_n^{2} S_n$. Then,
\begin{equation}\label{eq:CondInicial}
\begin{split}
\frac{\partial\Phi}{\partial r}(x,r) &=-\frac{\lambda_n^{2}}{N r^{N-1} \alpha(N)} \int_0^r \int_{\partial B_{\tilde{r}}(x)} S_n \,\diff \sigma_{\tilde{r}} \diff \tilde{r}\\
&= -\frac{\lambda_n^{2}}{r^{N-1}} \int_0^r \tilde{r}^{N-1} \Phi(\tilde{r}) \, \diff \tilde{r}.
\end{split}
 \end{equation}

Therefore we have for fixed $x$,
\begin{displaymath}
r^{N-1} \frac{\partial\Phi}{\partial r}(x,r) + \lambda_n^{2} \int_0^r \tilde{r}^{N-1} \Phi(x,\tilde{r}) \, \diff \tilde{r}=0,
\end{displaymath}
from which we obtain the equation for fixed $x$,
\begin{equation}\label{eq:EqDiffPhi}
\begin{cases}
r \frac{\partial^2\Phi}{\partial r^2}(x,r) + (N-1) \frac{\partial\Phi}{\partial r}(x,r)+ \lambda_n^{2} r \Phi(x,r)=0, & \\
\Phi(x,0)= S_n(x), &\\
\frac{\partial\Phi}{\partial r}(x,0)=0. & 
\end{cases}
\end{equation}

The boundary condition $\Phi^{\prime}(0)=0$ comes from \eqref{eq:CondInicial}. For fixed $x\in\Omega$, the domain of validity of this equation is the region $\{r\in\R : B_{|r|}(x)\subset \Omega\}$. Moreover, from \eqref{eq:SphericalMean}, we know that  $\Phi$ is an even function.
\end{proof}

The solution of \eqref{eq:EqDiffPhi} is given by
\begin{displaymath}
\Phi_n(x,r)=S_n(x) \GSol_{N}(r \lambda_n).
\end{displaymath}
For $N=1, 2, 3$, $\GSol_{N}$ has particularly recognizable forms: $\GSol_{1}(\cdot)=\cos(\cdot)$ (the Cosine function), $\GSol_{2}(\cdot)=J_0(\cdot)$ (the Bessel function of order $0$), $\GSol_{3}(\cdot)=\sinc(\cdot)$ (the Sinc function defined for any $x\in\R$ by $\sinc(x)=\frac{\sin(x)}{x}$). For all other values $N\geq 4$, by using the Frobenius method (see for instance \cite{Frobenius2012}), we obtain that 
\begin{equation}\label{eq:Frobenius}
\GSol_{N}(r \lambda)=1+\sum_{m=1}^{\infty} \frac{(-1)^{m}  (r \lambda)^{2m}  }{2^m (m!) \prod_{k=1}^{m} (N+2(k-1))}.
\end{equation} 
This alternating series represents an entire even function, and by making a comparison with the exponential function,  we see that for $\lambda\in\C$, there exists $C_N>0$ such that 
\begin{displaymath}
|\GSol_{N}(r \lambda)|< C_N \ee^{ |r| \mathrm{|img}( \lambda)|}.
\end{displaymath}
Comparing to a trigonometric series, we also see that for $r\lambda \in \R$, $G_N(r\lambda)$ is bounded in the real line. Then by the Paley-Wiener Theorem, there is a compactly supported distribution $\theta_{N,r}$  in the interval $[-r,r]$ such that for each fixed $r$,  $\Fourier_{\xi}(\theta_{N,r}(\xi))=\GSol_{N}(r \lambda)$. For the cases $N=1,2,3$, we can easily obtain explicitly  $\theta_{N,r}$.

For $N=1$, $\GSol_{1}(r \lambda)=\cos(r\lambda)$, and using a Fourier transform in the variable $\lambda$, we obtain that $\theta_{1,r}=\pi (\delta_{-r}+\delta_r)$. For $N=2$, $\GSol_{2}$ is the Bessel function of order $0$, $J_0$, and 
\begin{equation}\label{eq:DefTesteThetaN2}
\Fourier_{\lambda}(J_0(r \lambda))(\xi)= \frac{\sqrt{2} (1-\He(1-(r/\xi)^2))}{\sqrt{\pi (r^2-\xi^2)}}.
\end{equation}
Observe that the above function is integrable in the interval $]-r,r[$, since the growth of it near $r$ is of the order $\xi^{\frac{1}{2}}$.
For $N=3$, the Fourier transform of the Sinc function in $\R$ also has compact support, and is given by
\begin{equation}\label{eq:DefTesteThetaN3}
\Fourier_{\lambda}(\sinc(r \lambda))(\xi)= \frac{\pi}{r} \chi_{]-r,r[}(\xi),
\end{equation}
and is obviously integrable in the interval $]-r,r[$, and it is also an even function.

Let $\varphi\in \Fteste[B_{\epsilon}(0)]$. We test $S_n$ against $\varphi$ and take the spherical mean of the result around a point $x_0\in\Omega$ as follows
\begin{displaymath}
\Phi_{n,\varphi}(x_0,r)\doteq \frac{1}{N \alpha(N) r^{N-1}}  \int_{\partial B_r(x_0)} \int_{\R^N} S_n(y) \varphi(x-y)\,\diff y \, \diff\sigma(x), 
\end{displaymath}
when $r>0$ and $\epsilon>0$ are such that $B_{r+\epsilon}(x_0)\subset \Omega$. The result is
\begin{equation}
\begin{split}
\Phi_{n,\varphi}(x_0,r)&=- \left[ \int_{\R^N} S_n(x_0-\zeta) \varphi(\zeta)\, \diff\zeta \right]  \, \GSol_{N}(r \lambda_n)\\
&= K_{\varphi}(n,x_0) \GSol_{N}(r \lambda_n), \, N\geq 1,
\end{split}
\end{equation}
where $K_{\varphi}(n,x_0) =- \left[\int_{\R^N} S_n(x_0-\zeta) \varphi(\zeta)\, \diff\zeta \right] $.



Now we are ready for the proof of the main result of this article. 

\begin{proof}[(Proof of Theorem \ref{theo:Main})] We deal with all the cases $N\geq 1$ together, making some observations for the case $N=1$ where appropriate.
	
	 Take any arbitrary point $P\in \Omega$. If $P\in \omega$, then $u(P,t)=0$, $\forall t\in]-T,T[$ by hypothesis. Then we may suppose that $P\notin \omega$.  We will prove that there exists $t_P>0$ such that $u(P,t)=0$, $\forall t\in]-t_P,t_P[$. As illustrated in Figure \ref{fig:PolylineN1}, for $T>T_{max}$, there are always points $x_0\in \omega$, $x_1$, ..., $x_M=P$, and radius $t_0$, ..., $t_M$ such that $x_m$, $m=1,...,M$ are centers of balls $B_{t_m}(x_m) $ of diameter $2t_m$, these balls cover the path from $x_0$ to $x_M=P$, and $\sum_{m=0}^{M} t_m < T$. We call this juxtaposition of line segments $x_0x_1\ldots x_M$ a polyline.

	\begin{figure}
	\centering
	\begin{subfigure}[c]{0.6\linewidth}
		\psfrag{1}{$x_0$}
		\psfrag{2}{$P$}
		\psfrag{3}{$B_{T_0}(x_0)$}
		\psfrag{4}{$B_{T_1}(x_1)$}
		\psfrag{w}{$\omega$}
		\includegraphics[width=1.0\linewidth]{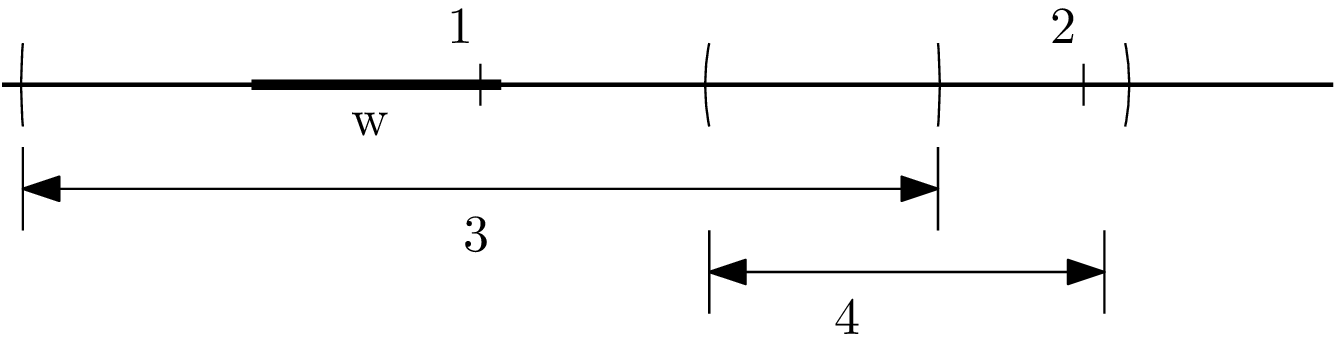}
		\caption{Polyline for $N=1$}
		\label{fig:PolylineN1}
	\end{subfigure}
	~
	\begin{subfigure}[c]{0.35\linewidth}
		\centering
		\psfrag{1}{$P$}
		\psfrag{w}{$\omega$}
		\includegraphics[width=1.0\linewidth]{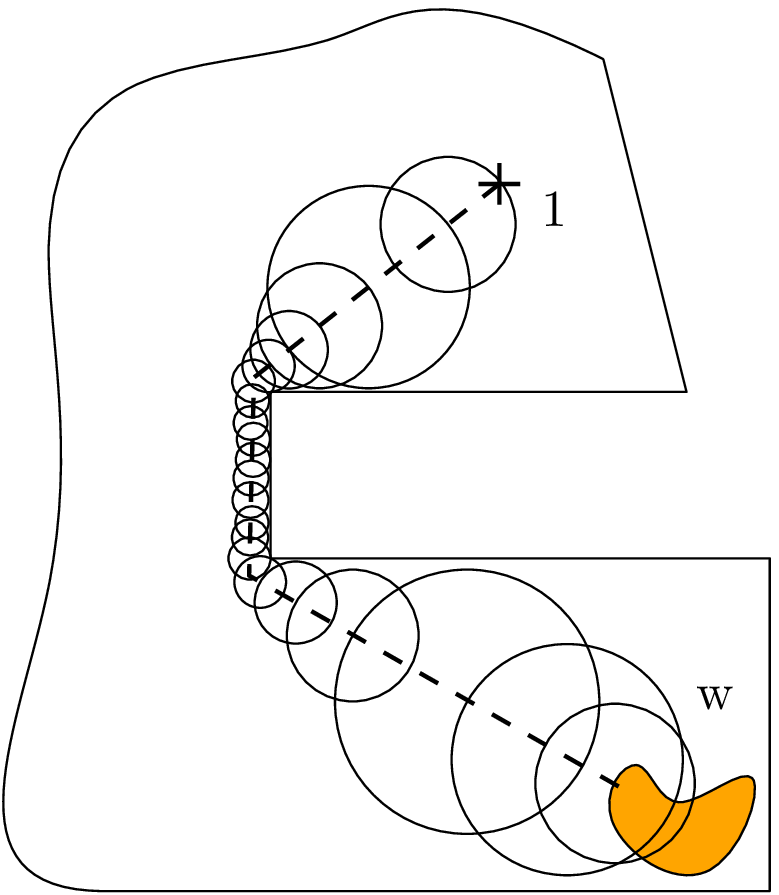}
		\caption{Polyline for $N=2$.}
		\label{fig:PolylineN2}
	\end{subfigure}
	\caption{Polylines}\label{fig:Polyline}
\end{figure}

Select $\epsilon>0$ small so that $B_{\epsilon}(x_0)\subset B_{T_0}(x_0) \cap \omega$ and chose $\varphi\in \Fteste[B_{\epsilon}(0)]$. 

By hypothesis, 	
\begin{equation}\label{eq:HypothesisWithTestFunction}
	\begin{split}
\mathcal{S}(x_0,t) & \doteq  \sum_{n=0}^{+\infty} a_n \pardual{S_n}{\varphi(x_0-\cdot)} \ee^{i \lambda_n t}\\
&= \sum_{n=0}^{+\infty} a_n K_{\varphi}(n,x_0) \ee^{i \lambda_n t}=0, \, \forall t\in]-T,T[.
	\end{split}
\end{equation}
	The change in the order of the integration and summation performed above is justified by the Monotone Convergence Theorem.
	
Now we take the spherical mean of the function $x \mapsto \mathcal{S}(x,t)$ around the sphere $\partial B_{r}(x_0)$, centered at $x_0\in\Omega$ with any radius $r>0$, so that $B_{\epsilon+r}(x_0)\subset B_{T_0}(x_0) \subset\Omega$. It becomes for $N=1$, 
	\begin{equation}\label{eq:SphericalMeanuTestadoN1}
	\frac{1}{2}(\mathcal{S}(x_0-r,t)+ \mathcal{S}(x_0+r,t)) =
	  \sum_{n=0}^{+\infty} a_n K_{\varphi}(n,x_0) \GSol_{1}(r \lambda_n) \ee^{i \lambda_n t}, \, \forall t\in ]-T,T[.
	\end{equation}
	For $N\geq 2$, it becomes
		\begin{equation}\label{eq:SphericalMeanuTestadoN23}
	\begin{split}
	&\frac{1}{N \alpha(N) r^{N-1}} \int_{\partial B_{r}(x_0)} \pardual{u(t,\cdot)}{\varphi(x-\cdot)} \, \diff \sigma(y)= \\
	& \quad\quad\quad\quad  \quad\quad\quad\quad = \sum_{n=0}^{+\infty} a_n K_{\varphi}(n,x_0) \GSol_{N}(r \lambda_n) \ee^{i \lambda_n t}, \, \forall t\in ]-T,T[.
	\end{split}
	\end{equation}

	Of course, directly from \eqref{eq:HypothesisWithTestFunction} we also have that $\forall t_0\in]-T_0,T_0[$ and $\forall t_c\in ]-(T-T_0),T-T_0[$,
\begin{equation}\label{eq:FuncamentalSplitSum}
\mathcal{S}(x_0,t_0+t_c)=\sum_{n=0}^{+\infty} a_n K_{\varphi}(n,x_0) \ee^{i \lambda_n (t_0+t_c)}=0.
\end{equation}

For $N=1$, $\forall t_c\in ]-(T-T_0),T-T_0[$,
\begin{displaymath}
0=\frac{1}{2}(\mathcal{S}(x_0,t_0+t_c)+\mathcal{S}(x_0,-t_0+t_c))=\sum_{n=0}^{+\infty} a_n K_{\varphi}(n,x_0) \GSol_{1}(r \lambda_n) \ee^{i \lambda_n t_c}.
\end{displaymath}
That is, the mean in \eqref{eq:SphericalMeanuTestadoN1} is null for $\forall t\in ]-(T-T_0),T-T_0[$.

For $N=1$, there are no repeated eigenvalues in the list $(\lambda_n^2)_{n\in\N}$. But for $N\geq 2$ we must deal with this possibility. 

\begin{description}
	\item[No repeated values in $\bm{(\lambda_n^2)_{n\in\mathbb{\bf N}}}$. ]
	For $N\geq 2$, we continue to regard expression in \eqref{eq:FuncamentalSplitSum} as a function of $t_0$, and test it against $\theta_{N,r}$, which has support contained in $[-r,r]$, $r>0$, so that $B_{\epsilon+r}(x_0)\subset B_{T_0}(x_0) \subset\Omega$. We obtain
\begin{eqnarray}
0=\pardual{\sum_{n=0}^{+\infty} a_n K_{\varphi}(n,x_0) \ee^{i \lambda_n (\cdot +t_c)}}{\theta_{N,r}} &=&\sum_{n=0}^{+\infty} a_n K_{\varphi}(n,x_0) \GSol_N(r \lambda_n) \ee^{i \lambda_n t_c}, \nonumber\\
& & \quad \forall t_c\in ]-(T-T_0),T-T_0[. \nonumber
\end{eqnarray}

	Recalling \eqref{eq:SphericalMeanuTestadoN23}, this means that the spherical mean of $x \mapsto\pardual{u(t,\cdot)}{\varphi(x-\cdot)}$ over a sphere centered at $x_0$ with any radius $r>0$ such that $B_{r+\epsilon}(x_0)\subset B_{T_0}(x_0)$ is zero, for $t\in ]-(T-T_0),T-T_0[$. 
	
		Now, by Lemma \ref{teo:AchadoZalcmanCasoN1}  for $N=1$, and by Lemma \ref{teo:AchadoZalcman} for $N\geq 2$, we have that  $x \mapsto\pardual{u(t,\cdot)}{\varphi(x-\cdot)}$ is the null function for $x\in B_{T_0-\epsilon}(x_0)$, for $t\in ]-(T-T_0),T-T_0[$. By shrinking $\epsilon>0$, we conclude that  $x \mapsto\pardual{u(t,\cdot)}{\varphi(x-\cdot)}$ is the null function for $x\in B_{T_0}(x_0)$, for $t\in ]-(T-T_0),T-T_0[$. Since $\varphi\in \Fteste[B_{\epsilon}(0)]$ is arbitrary, we conclude that the restriction of $u$ to the set $ ]-(T-T_0),T-T_0[\times B_{T_0}(x_0)$ is the null distribution, that is,  $u|_{ ]-(T-T_0),T-T_0[\times B_{T_0}(x_0)}=0$. 
		
		For the next step, take $\omega \subset B_{T_0}(x_0)\cap B_{T_1}(x_1)$ and repeat the procedure used above. 
		
		By induction, we proceeding  along the  polyline $x_0x_1\ldots x_m\subset \Omega$, we reach the conclusion that $u|_{ ]-T_M,T_M[\times B_{T_M}(x_M)}=0$. In this case, $t_P=T_M$, $P=x_M$.
	
	This ends the proof for the  case when there are no repeated values in the list $(\lambda_n^2)_{n\in\N}$.

\item[There are repeated values in $\bm{(\lambda_n^2)_{n\in\mathbb{\bf N}}}$.] The number of repetitions of each value in $(\lambda_n)_{n\in\N}$ is their respective geometric multiplicity, which are all finite. The proof for this case is essentially the same. After applying the same steps above to the series
\begin{equation}\label{eq:mainSeriesRepeatedEigen}
u(t,x)=\sum_{n=0}^{+\infty} \left[\sum_{m=1}^{g_n} a_{n,m} S_{n,m}(x)\right] \ee^{i \lambda_n t},
\end{equation} 
where $g_n$ is the geometric multiplicity of the eigenvalue $\lambda_n^2$, we come to the same conclusion that for any $P\in \Omega$, there is a $t_P > 0$ such that $u(P,t)=0$, for $t\in ]-t_P,t_P[$.
\end{description}

For both cases, to finish the proof, fix any $\in ]-t_P,t_P[$ and use the fact the the eigenvectors form a complete Hilbert basis of $\HZ{1}{\Omega}$ to conclude that all coefficients that appear in the distribution $u$ are null.

With this, we finish the proof of our main result.
\end{proof}

\begin{corollary}\label{teo:CorolarioTodosLambdas}
	Given any non empty open set $\omega\subset\Omega\subset \R^N$, $N\geq 1$, if $T>T_{max}(\Omega,\omega)$ and  $u|_{\omega\times ]-T,T[}=0$, for $u$ given by 
	\begin{displaymath}
	u(t,x)=\sum_{n=0}^{+\infty} \left(a_n \ee^{-i \lambda_n t} + b_n \ee^{+i \lambda_n t}\right) S_n(x),
	\end{displaymath}
	 for $(\lambda_n)_{n\in\N}$ and $(S_n)_{n\in\N}$ as in Theorem \ref{theo:Main},  then $u\equiv 0$.
\end{corollary}
\begin{proof}
	The proof follows almost exactly the same as in Theorem  \ref{theo:Main}. It suffices to observe that for $N\geq 2$, $\theta_{N,r}$  is an even function and for $N=1$,  $(\ee^{i \lambda_n t}+\ee^{-i \lambda_n t})/2=\cos(\lambda_n t)=G_1(\lambda_n t)$.
\end{proof}

\section{Application to the unique continuation of solutions of the wave equation}\label{sec:ApplicationWave}
In this section we show two applications involving the wave equation.  

For both applications below, let $\Omega\subset \R^N$, $N\geq 1$, be an open bounded and connected set with continuous piecewise $\Cont{\infty}$ boundary, and $\omega\subset \Omega$ be any non empty set.  Let also $T>T_{max}(\Omega,\omega)$.
 
 Following the notation used in the previous sections, $(S_n)_{n\in\N}$ and  $(S_n/\lambda_n)_{n\in\N}$, forms an orthonormal Hilbert basis for $\Le{2}(\Omega)$ and $\HZ{1}{\Omega}$ respectively. Any distribution in $\Sobo{-1}{\Omega}$ can be expressed as $\sum_{n\in\N} C_n \lambda_n S_n$ with $(C_n)_{n\in\N}\in \lle{2}$, because of the Riesz Representation Theorem. In fact, let $F\in \Sobo{-1}{\Omega}$. By this theorem, there is a unique $f\in \HZ{1}{\Omega}$ such that 
 \begin{displaymath}
 F(\phi)=\pardual{\phi}{f}_{\HZ{1}{\Omega}}=\int_{\Omega}\sum_{n=1}^N \frac{\partial \phi}{\partial x_n}\frac{\partial f}{\partial x_n} \,\diff x=
 -\pardual{\phi}{\triangle f}_{\Le{2}(\Omega)}.
 \end{displaymath}
 Now, $f=\sum_{n\in\N} C_n \frac{S_n}{\lambda_n}$ for $(C_n)_{n\in\N}\in \lle{2}$, and $\triangle S_n=-\lambda_n^2$. Then,
 \begin{displaymath}
 F(\phi)= \pardual{\phi}{\sum_{n\in\N} C_n \lambda_n S_n}_{\Le{2}(\Omega)},
 \end{displaymath}
 showing that indeed $F\in \Sobo{-1}{\Omega}$ can be represented as $\sum_{n\in\N} C_n \lambda_n S_n$ with $(C_n)_{n\in\N}\in \lle{2}$.
 
 \begin{description}
 	\item[First problem.]  Several researchers have analyzed the problem of unique continuation for the solutions of the wave equation, starting from Ruiz  \cite{RuizWave92}. The scientific development and recent results concerning this problem can be found in  \cite{Bosi2018} \cite{Joly2013} and \cite{Joly2019}.  More related to the applications we are going to show in this section, we mention  \cite[Theorem 6.1]{Joly2019} (also see \cite[Corollary 3.2]{Joly2013}). 
 	
  Consider the problem concerning the wave equation 
\begin{equation}\label{eq:waveProb1}
\left\{
\begin{array}{ll}
	\partial_{t}^{2} w(x, t)-\Delta w(x, t)=0 & {\text { in } ]-T,T[ \times\Omega}, \\ 
	w(x, t)=0 & {\text { on } [0,+\infty)\times\partial\Omega}, \\ 
	w(x, 0)=w_0 & {\text { in } \Omega},\\
	\partial_{t} w(x, 0)=v_0 & {\text { in } \Omega},
\end{array}
\right.
\end{equation} 
for $w_0\in \HZ{1}{\Omega}$ and $v_0\in\Le{2}(\Omega)$. Below in the observation \ref{obs:Regularity} we comment about the possibility of using more irregular spaces.

Clearly, the initial conditions $w_0\in \HU{\Omega}$ and $v_0\in\Le{2}(\Omega)$ can be represented by $w_0=\sum_{n\in\N} \frac{A_n}{\lambda_n} S_n$ with $(A_n)_{n\in\N}\in \lle{2}$ and  $v_0=\sum_{n\in\N} B_n  S_n$, for $(B_n)_{n\in\N}\in \lle{2}$.

The unique solution of  \eqref{eq:waveProb1} with $w\in \Le{2}(]-T,T[,\HZ{1}{\Omega})$, $\partial_t w \in \Le{2}(]-T,T[,\Le{2}(\Omega))$ and $\partial_{tt}w  \in  \Sobo{-1}{\Omega})$ is given by
\begin{displaymath}
w(t,x)=\sum_{n=0}^{+\infty} \frac{1}{\lambda_n} \left[A_n \cos(\lambda_n t)  + B_n  \sin(\lambda_n t) \right] S_n(x),
\end{displaymath}
that can be rewritten as
\begin{equation}\label{eq:solProb1}
w(t,x)=\sum_{n=0}^{+\infty} \left(a_n \ee^{-i \lambda_n t} + b_n \ee^{+i \lambda_n t}\right) S_n(x),
\end{equation}
where $b_n=\frac{1}{2 \lambda_n}\left(A_n -i B_n \right)$  and $a_n=\frac{1}{2 \lambda_n}\left(A_n+i B_n \right)$ for $n\in\N \cup \{\ 0 \}$.

Now, a simple application of Corollary \ref{teo:CorolarioTodosLambdas} shows that  $w\equiv 0$. In this particular example, we can also conclude that $w_0$ and $v_0$ are both determined uniquely by the data $w|_{]-T, T[\times\omega}$. This conclusion is in fact valid for for any $\omega$ such that $T>T_{max}(\Omega,\omega)$.

\begin{remark}\label{obs:Regularity}
Note that by using the interpretation provided by \eqref{eq:DefU}, the problem can be stated in more irregular spaces. For instance, $w_0\in \Le{2}(\Omega)$ and $v_0\in\Sobo{-1}{\Omega}$. In this case, the solution $w\in \Le{2}(]-T,T[,\Le{2}(\Omega))$ would be represented as in \eqref{eq:solProb1} but with $a_n=\frac{1}{2}\left(A_n+i B_n \right)$ and $b_n=\frac{1}{2}\left(A_n -i B_n \right)$, $\forall n\in\N \cup \{\ 0 \}$. 
The boundary conditions still possess a meaning via \eqref{eq:DefU}.
\end{remark}

\item[Second problem. ] Consider the problem
\begin{equation}\label{eq:waveProb2}
\left\{
\begin{array}{ll}
\partial_{t}^{2} w(x, t)-\Delta w(x, t)=g(t)\, f(x) & {\text { in } ]0,T[ \times\Omega}, \\ 
w(x, t)=0 & {\text { on } [0,+\infty)\times\partial\Omega}, \\ 
w(x, 0)=\partial_{t} w(x, 0)=0 & {\text { in } \Omega},
\end{array}
\right.
\end{equation} 
for $g\in\Cont{1}([0,T[)$, $g(0)\neq 0$, $f\in\Sobo{-1}{\Omega}$. We are going to show that if $w|_{[0, T[\times\omega}=0$, then $f\equiv 0$. A related result can be found in \cite{Cheng2002}.

The function $f$ can be expressed as $f=\sum_{n\in\N} C_n \lambda_n S_n$, for $(C_n)_{n\in\N}\in \lle{2}$. The solution $w\in \Le{2}(]0,T[,\HZ{1}{\Omega})$ with $\partial_t w \in \Le{2}(]-T,T[,\Le{2}(\Omega))$ and $\partial_{tt}w  \in  \Sobo{-1}{\Omega})$ of \eqref{eq:waveProb2} is given by
\begin{displaymath}
w(t,x)=\int_{0}^{t} g(t-\tau) \sum_{n\in \N} C_n \sin(\lambda_n t) S_n(x) \, \diff\tau.
\end{displaymath}

Since $w|_{[0, T[\times\omega}=0$, by taking the first derivative of the last equation with respect to $t$, we obtain the following Volterra equation of the second kind.
\begin{displaymath}
0=g(0) \sum_{n\in \N} C_n \sin(\lambda_n t) S_n(x) + \int_{0}^{t} g^\prime(t-\tau) \sum_{n\in \N} C_n \sin(\lambda_n t) S_n(x) \, \diff\tau,
\end{displaymath}
for all $t\in [0,T[$, from which we conclude that 
\begin{displaymath}
\sum_{n\in \N} C_n \sin(\lambda_n t) S_n(x) =0, \, \forall t\in [0,T[,
\end{displaymath}
which can be extended in a natural way to $]-T,T[$ oddly. Now, an application of  Corollary \ref{teo:CorolarioTodosLambdas} shows that  $w\equiv 0$ and $f\equiv 0$.
 \end{description}

\section{Conclusions} \label{Conclusions}
We proved Theorem \ref{theo:Main}, that concerns an almost periodic distribution in the time variable $t$, given by a series whose coefficients depend on the space variable $x$. This result and possible adaptations have many applications in Mathematical Physics, for example inverse problems involving the plate equation in bounded domains with various boundaries. Two of them was given in section \ref{sec:ApplicationWave}, where we considered the problem of unique continuation for the wave equation.

\section*{Acknowledgments}
The author profoundly thanks Professor Paulo D. Cordaro of Instututo de Matemática e Estatística da Universidade de São Paulo for his constant support and enlightening guidance. 

\bibliographystyle{amsplain}
\bibliography{References}

\end{document}